\documentclass[leqno]{amsart}
\usepackage{amsmath}
\usepackage{mathtools}
\usepackage{amssymb}
\usepackage{amsthm}
\usepackage{graphicx}
\usepackage{enumerate}
\usepackage[mathscr]{eucal}
\usepackage{upgreek}
\usepackage{xcolor}
\theoremstyle{plain}
\newtheorem{theorem}{Theorem}[section]
\newtheorem{prop}[theorem]{Proposition}
\newtheorem{cor}{Corollary}[theorem]

\theoremstyle{definition}

\newtheorem{remark}{Remark}[section]

\newtheorem{example}{Example}[theorem]

\usepackage[pagewise]{lineno}

\begin{document}

\title[Orthogonality of bilinear forms]{Orthogonality of bilinear forms and application to matrices}
                \newcommand{\acr}{\newline\indent}

\author{Saikat Roy, Tanusri Senapati, Debmalya Sain}

\address[Roy]{Department of Mathematics\\ National Institute of Technology Durgapur\\ Durgapur 713209\\ West Bengal\\ INDIA}
\email{saikatroy.cu@gmail.com}

\address[Senapati] {Department of Mathematics\\Gushkara Mahavidyalaya\\ West Bengal\\ India}
\email{senapati.tanusri@gmail.com}

\address{(Sain)~Department of Mathematics, Indian Institute of Science, Bengaluru 560012, Karnataka, India}
\email{saindebmalya@gmail.com}          


\thanks{Dr. Debmalya Sain feels elated to acknowledge the heartening friendship Mr. Sujoy Koner. The research of Mr. Saikat Roy is supported by CSIR MHRD in the form of Junior Research Fellowship under the supervision of Prof. Satya Bagchi.}

\begin{abstract}
We characterize Birkhoff-James orthogonality of continuous vector-valued functions on a compact topological space. As an application of our investigation, Birkhoff-James orthogonality of real bilinear forms are studied. This allows us to present an elementary proof of the well-known Bhatia-\v{S}emrl Theorem in the real case.
\end{abstract}

\maketitle

\section{Introduction}

Throughout the text, the letter $ \mathbb{X} $ denotes a normed linear space and the symbols  $ \mathbb{H}_1,~ \mathbb{H}_2 $ are reserved for finite-dimensional Hilbert spaces. In this article, we work only with real normed linear spaces. Let $ S_{\mathbb{X}} $ denote the unit sphere of $ \mathbb{X}, $ i.e., $ S_{\mathbb{X}} = \{ x \in \mathbb{X} : \|x\|=1 \}. $ Let $ \mathcal{U} $ be a compact topological space and let $ \mathcal{C}(\mathcal{U}, \mathbb{X}) $ denote the real normed linear space of all vector-valued continuous functions from $ \mathcal{U} $ to $ \mathbb{X} $, equipped with the norm $ \| \cdot \|_\mathcal{U} $, where
\begin{align*}
\|f\|_\mathcal{U} = \underset{u\in \mathcal{U}}{\sup}\|f(u)\|~~\forall~~f\in \mathcal{C}(\mathcal{U}, \mathbb{X}).
\end{align*}
For any $ f\in \mathcal{C}(\mathcal{U}, \mathbb{X}) $, we denote the norm attainment set of $ f $ by $ \mathcal{M}_f $, i.e.,
\begin{align*}
 \mathcal{M}_f = \{ u\in \mathcal{U} :  \|f(u)\| = \|f\|_\mathcal{U} \}.
\end{align*} 
We note that $ \mathcal{M}_f \neq \emptyset $ for any $ f \in \mathcal{C}(\mathcal{U}, \mathbb{X}) $. Let $ \mathbb{L}(\mathbb{X}) $ denote the normed linear space of all bounded linear operators defined on $ \mathbb{X} $, endowed with the usual operator norm. For any $T\in  \mathbb{L}(\mathbb{X}) $, the norm attainment set of $T$ is defined as $ M_T = \{ x \in S_\mathbb{X} : \|Tx \| = \|T\| \}. $ Observe that for any $ T \in \mathbb{L}(\mathbb{X}) $, where $ \mathbb{X} $ is a finite-dimensional Banach space, the restriction of $ T $ on $ S_{\mathbb{X}} $ (which is again denoted by $ T $ for the sake of simplicity) is a member of $ \mathcal{C}(S_\mathbb{X}, \mathbb{X} ) $ and $ \|T\| = \|T\|_{S_\mathbb{X}} $. \\

Let $ \mathcal{B}(\mathbb{H}_1\times \mathbb{H}_2, \mathbb{R}) $ denote the Banach space of all real bilinear forms on $ \mathbb{H}_1 \times \mathbb{H}_2, $ endowed with the supremum norm. Let $ \langle \cdot , \cdot \rangle $ denote the usual inner product in $ \mathbb{H}_1$, i.e., $ \langle x_1, x_2 \rangle = x_1^Tx_2 $ for all $ x_1,x_2\in \mathbb{H}_1$, where $x_1^T$ stands for the transpose of $x_1$. Assuming that dim $\mathbb{H}_1=m$ and dim $\mathbb{H}_2=n$, it is well-known that each $m\times n$ matrix $ A$  can be identified with a unique linear operator in $\mathbb{L}(\mathbb{H}_2,\mathbb{H}_1)$. Moreover,  each $m\times n$ matrix $ A $ induces a unique bilinear form $ \mathcal{T}_A $, defined by, 
\begin{align}\label{presentation}
\mathcal{T}_A(x,y) = \langle x, Ay \rangle~~\forall~(x,y)\in \mathbb{H}_1\times \mathbb{H}_2.
\end{align}
On the other hand, every bilinear form in $ \mathcal{B}(\mathbb{H}_1\times \mathbb{H}_2, \mathbb{R}) $ corresponds to a unique $m\times n$ matrix. Throughout the article, we use the symbol $ \mathcal{T}_A $ to denote a member of $ \mathcal{B}(\mathbb{H}_1\times \mathbb{H}_2, \mathbb{R}) $ of the form (\ref{presentation}). For any $ \mathcal{T}_A\in \mathcal{B}(\mathbb{H}_1\times \mathbb{H}_2, \mathbb{R}) $, the norm of $ \mathcal{T}_A $ is defined by 
\begin{align*}
\|\mathcal{T}_A\| = \underset{(x,y)\in S_{\mathbb{H}_1}\times S_{\mathbb{H}_2}}{\sup}|\mathcal{T}_A(x,y)|.
\end{align*}
The norm attainment set of $\mathcal{T}_A$ is defined as $M_{\mathcal{T}_A} = \{ (x,y) \in S_{\mathbb{H}_1} \times S_{\mathbb{H}_2} : \| \mathcal{T}_A(x,y) \| = \| \mathcal{T}_A \| \}. $ Let us also observe that $ \mathcal{T}_A\in \mathcal{B}(\mathbb{H}_1\times \mathbb{H}_2, \mathbb{R}) $ implies that $ \mathcal{T}_A \in \mathcal{C}(S_{\mathbb{H}_1}\times S_{ \mathbb{H}_2}, \mathbb{R}) $ and $ \|\mathcal{T}_A\|_{S_{\mathbb{H}_1} \times S_{\mathbb{H}_2}} = \| \mathcal{T}_A \| $.\\  

Given any two elements $ x,y \in \mathbb{X}$, $ x $ is said to be  Birkhoff-James orthogonal to $ y, $ written as $ x \perp_B y, $   if $\|x+\lambda y\| \geq \|x\| $ for all $ \lambda \in \mathbb{R};$ see \cite{B, J, Ja}. Similarly, for any two operators $ T,~A\in \mathbb{L}(\mathbb{X}) $, $ T\perp_B A $, if $\|T+\lambda A\| \geq \|T\| $ for all $ \lambda \in \mathbb{R}.$ In \cite{BS}, Bhatia and $ \breve{S}$emrl completely characterized Birkhoff-James orthogonality of matrices identified as linear operators defined on a finite-dimensional Hilbert space. Later on, Sain \cite{S} extended the said result to characterize Birkhoff-James orthogonality of linear operators on a finite-dimensional Banach space $\mathbb{X}$. To obtain the desired result, the author introduced the notions of $ x^+ $ and $ x^- $ for any non-zero element $ x\in \mathbb{X} $, where  $x^+ =\{y\in \mathbb{X}: \| x+\lambda y \| \geq \| x \| ~\forall \lambda\geq 0\}$ and  $x^- =\{y\in \mathbb{X}: \| x+\lambda y \| \geq \| x \| ~\forall \lambda\leq 0\}$. We refer the readers to \cite{S} for more information in this regard.\\

The aim of the present article is to assimilate the above mentioned works in the more general setting of $ \mathcal{C}(\mathcal{U}, \mathbb{X}) $. As a consequence of this, we completely characterize Birkhoff-James orthogonality of real bilinear forms. As an important application, we provide an elementary proof of the classical  Bhatia-\v{S}emrl Theorem in the real case, using Birkhoff-James orthogonality of real bilinear forms. Moreover, we observe that some well-known characterizations of Birkhoff-James orthogonality \cite{S, SP}  of linear operators defined between finite-dimensional real Banach spaces follow immediately from our obtained results. 

\section{main result}

We begin the section with a complete characterization of Birkhoff-James orthogonality in $ \mathcal{C}(\mathcal{U}, \mathbb{X}) $.

\begin{theorem}\label{realcontinuity}
Let $ \mathcal{U} $ be a compact topological space and let $ \mathbb{X} $ be a normed linear space. Let $ f,g\in \mathcal{C}(\mathcal{U}, \mathbb{X}) $ be non-zero. Then $ f\perp_B g $ if and only if  there exist $ u_1,u_2\in \mathcal{M}_f $ such that $ g(u_1) \in f(u_1)^+$ and $ g(u_2) \in f(u_2)^- $.
\end{theorem}

\begin{proof}
We only prove the necessary part of the theorem as the sufficient part follows directly. Suppose on the contrary that $ g(u) \in f(u)^- $ for all $ u \in \mathcal{M}_f $. Let $ u_0 \in \mathcal{M}_f $ be arbitrary. Then  there exists $ \lambda > 0 $ such that $ \| f(u_0) + \lambda  g(u_0) \| < \| f \|_\mathcal{U} $. By the convexity of the norm, we obtain that
\begin{align}\label{separation}
\| f(u_0) + t\lambda  g(u_0) \| < \| f \|_\mathcal{U} ~\forall~ t \in (0,1). 
\end{align}
Now, we consider the continuous function $ \psi : \mathcal{U} \times [-1,1] \rightarrow \mathbb{R} $ defined by 
\begin{align*}
\psi(u, \mu) = \| f(u) + \mu  g(u) \| ~\forall~(u,\mu)\in \mathcal{U} \times [-1,1].
\end{align*}
Then from (\ref{separation}) it follows that there exists $ \mu_0\in (0,1) $ such that $ \psi(u_0, \mu_0) < \|f \|_\mathcal{U} $. Since $ \psi $ is continuous, we can find  $ \delta_0>0 $ and an open set $ \mathcal{G}_0 $ containing $ u_0 $ in $ \mathcal{U} $ such that
\begin{align*}
\psi(u, \mu) < \|f \|_\mathcal{U} ~ \forall~(u,\mu)\in \mathcal{G}_0\times (\mu_0-\delta_0, \mu_0+\delta_0).
\end{align*}
Once again, applying (\ref{separation}) in the above expression, we obtain
\begin{align*}
\psi(u, \mu) < \|f \|_\mathcal{U} ~ \forall~(u,\mu)\in \mathcal{G}_0\times (0,\mu_0).
\end{align*}
Next, let $ v_0\in \mathcal{U}\setminus \mathcal{M}_f $ be arbitrary. Then we have $ \psi(v_0,0) < \|f\|_\mathcal{U}. $ Again, using the continuity of $ \psi $, we can find $ \sigma_0 > 0 $ and an open set $ \mathcal{\hat G}_0 $ containing $ v_0 $ in $ \mathcal{U} $ such that
\begin{align*}
\psi(v, \sigma) < \|f \|_\mathcal{U} ~ \forall~(v,\sigma)\in \mathcal{\hat G}_0\times (-\sigma_0,\sigma_0).
\end{align*}
Clearly, the collection $ \mathcal{G} = \left\lbrace \mathcal{G}_0 : u_0\in \mathcal{M}_f \right\rbrace \cup \left\lbrace \mathcal{\hat G}_0 : v_0\in \mathcal{U}\setminus \mathcal{M}_f \right\rbrace  $ forms an open cover for $ \mathcal{U} $. Since $ \mathcal{U} $ is compact, there exist natural numbers $ k_1, k_2 $ such that
\begin{align*}
\mathcal{U} \subseteq \left(\bigcup\limits_{i=1}^{k_1}\left\lbrace \mathcal{G}_i : u_i\in M_f \right\rbrace\right)\cup \left(\bigcup\limits_{j=1}^{k_2}\left\lbrace \mathcal{\hat G}_j : v_j\in \mathcal{U}\setminus \mathcal{M}_f \right\rbrace\right).
\end{align*}
Now, fix some $ \lambda_0 \in \left(\bigcap\limits_{i=1}^{k_1} (0, \mu_i)\right)\cap \left(\bigcap\limits_{j=1}^{k_2}(-\sigma_j, \sigma_j)\right).$ Clearly, $ f + \lambda_0  g \in \mathcal{C}(\mathcal{U}, \mathbb{X}) $. Let $ w_0\in \mathcal{M}_{f + \lambda_0  g } $. However, it follows from the choice of $ \lambda_0 $ that 
\begin{align*}
\| f+ \lambda_0  g \|_\mathcal{U} = \| f(w_0) + \lambda  g(w_0) \| < \| f \|_\mathcal{U},
\end{align*}
a contradiction. Similarly, if we consider $ g(u) \in f(u)^+$ for all $ u \in \mathcal{M}_f $, then we once again arrive at a contradiction. This completes the proof of the theorem.
\end{proof}

The above characterization takes a special form whenever we assume that $\mathcal{M}_f $ is a connected subset of $ \mathcal{U} $ for any $ f\in \mathcal{C}(\mathcal{U}, \mathbb{X})$.

\begin{theorem}\label{connectedcontinuity}
Let $ \mathcal{U} $ be a compact topological space and let $ \mathbb{X} $ be a normed linear space. Let $ f\in \mathcal{C}(\mathcal{U}, \mathbb{X}) $ be such that $ f $ is non-zero and $ \mathcal{M}_f $ is connected. Then for any $ g \in \mathcal{C}(\mathcal{U}, \mathbb{X}) $, $ f\perp_B g $ if and only if there exists $ u_0\in \mathcal{M}_f $ such that $ f(u_0) \perp_B g(u_0) $.
\end{theorem}

\begin{proof}
We only prove the necessary part of the theorem as the sufficient part of the theorem follows trivially. Suppose on the contrary that $ f(u_0) \not\perp_B g(u_0) $ for any $ u_0\in \mathcal{M}_f $. Therefore, for each $ u\in \mathcal{M}_f $, either $ g(u) \in f(u)^+ $ or $ g(u)\in f(u)^- $. Now, we consider the following collections:
\begin{align*}
C_1 = \{ u\in \mathcal{M}_f : \| f(u) + \lambda g(u) \| > \|f\|_\mathcal{U},~ \lambda > 0 \},\\
C_2 = \{ u\in \mathcal{M}_f : \| f(u) + \lambda g(u) \| > \|f\|_\mathcal{U}, ~\lambda < 0 \}.
\end{align*}
We claim that $ C_1 \cup C_2 = \mathcal{M}_f $. Let $ v \in \mathcal{M}_f $ be arbitrary. Since $ f(v) \not\perp_B g(v) $, there exists $ \mu \in \mathbb{R} $ such that $ \| f (v)+ \mu g(v) \|< \|f \|_\mathcal{U} $. Without loss of generality, let $ \mu < 0 $, i.e., $ v\not\in C_2 $. Let $ \lambda > 0 $ be arbitrary. Then there exists $ t \in (0,1) $ such that $ f(v) = t (f(v) + \mu g(v)) + (1-t)(f(v)+ \lambda g(v) ). $ This proves that $ \| f \|_\mathcal{U} < \|f(v) + \lambda g(v)\|$, i.e., $ v \in C_1 $. Similarly, assuming $ \mu > 0 $, it can be shown that $ v \in C_2 $. Therefore, $ C_1 \cup C_2 = \mathcal{M}_f $, as desired.\\

Next, we claim that both $ C_1, C_2 $ are non-empty. Indeed, following the same arguments used in the proof of Theorem \ref{realcontinuity}, one can easily verify that $ C_1, C_2 \neq \emptyset $. Finally, we claim that $ C_1, C_2 $ form a separation of $ \mathcal{M}_f $. To establish our claim, we observe that $ \bar{C_1} \cap C_2 = \emptyset $ and $ C_1 \cap \bar{C_2} = \emptyset $, as otherwise we can find $ u_0\in \mathcal{M}_f $ such that $ f(u_0) \perp_B g(u_0) $. However, this is a contradiction to the fact that $ \mathcal{M}_f $ is connected. This completes the proof of the theorem. 
\end{proof}

The above theorem may not be true  if $ \mathcal{M}_f $ is not connected. We furnish the following example in support of this fact.

\begin{example}
Let $ \mathcal{U} = [0,2] \subset \mathbb{R} $ and $ \mathbb{X} = \mathbb{R} $. We define $ f, g \in \mathcal{C}(\mathcal{U}, \mathbb{X}) $ by $ f(u) = sin~ \pi u $ and $ g(u) = 1 $ for all $ u \in [0,2] $. Clearly, $ \mathcal{M}_f = \{ \frac{1}{2}, \frac{3}{2} \} $. Now, $ f(\frac{1}{2})g(\frac{1}{2}) > 0 $, i.e., $ g(\frac{1}{2}) \in f(\frac{1}{2})^+ $ and $ f(\frac{3}{2})g(\frac{3}{2}) < 0 $, i.e., $ g(\frac{3}{2}) \in f(\frac{3}{2})^- $. Then from Theorem \ref{realcontinuity}, it follows that $ f\perp_B g $, however, $ g(v) \neq 0 $ for any $ v \in \mathcal{M}_f $, i.e., $ f(v) \not \perp_B g(v) $ for any $ v \in \mathcal{M}_f $. 
\end{example}

Our next theorem, in some sense, assimilates the above two theorems. This turns out to be an important tool in our further developments. We omit the proof as it is trivial in view of Theorem \ref{realcontinuity} and Theorem \ref{connectedcontinuity}.

\begin{theorem}\label{componentcontinuity}
Let $ \mathcal{U} $ be a compact topological space and let $ \mathbb{X} $ be  a normed linear space. Let $ f,g\in \mathcal{C}(\mathcal{U}, \mathbb{X}) $ be non-zero. Let $ \mathcal{D}_f \subseteq \mathcal{M}_f $ be such that $ \mathcal{D}_f $ is connected. Then the following conditions are equivalent:\\
$ (i) $ There exist $ u_1, u_2 \in \mathcal{D}_f $ such that $ g(u_1) \in f(u_1)^+ $ and $ g(u_2) \in f(u_2)^- $.\\
$ (ii) $ There exists $ u_0\in \mathcal{D}_f $ such that $ f(u_0) \perp_B g(u_0) $.
\end{theorem}

As mentioned in the introduction, we obtain some of the earlier results on the Birkhoff-James orthogonality of linear operators as simple corollaries  to Theorem \ref{realcontinuity} and Theorem \ref{connectedcontinuity}.  

\begin{cor}\cite[Theorem 2.2]{S}\label{linearoperator}
Let $ \mathbb{X} $ be a finite-dimensional real Banach space. Let $ T, A\in \mathbb{L}(\mathbb{X}) $. Then $ T \perp_B A $ if and only if there exist $ x,y\in M_T $ such that $ Ax\in Tx^+ $ and $ Ay\in Ty^- $.
\end{cor}

\begin{proof}
The proof of the sufficient part is trivial. To prove the necessary part, we first observe that
\begin{align*}
\|T+\lambda A\|_{S_\mathbb{X}} = \|T+\lambda A\| \geq \|T\| = \|T\|_{S_\mathbb{X}},
\end{align*}
for all $ \lambda \in \mathbb{R} $. In other words, $ T ,A \in  \mathcal{C}(S_\mathbb{X}, \mathbb{X}) $ (by considering the respective restriction operators on $ S_{\mathbb{X}} $) with $ T \perp_B A $. Therefore, it follows from Theorem \ref{realcontinuity} that there exist $ x, y \in \mathcal{M}_T $ such that $ Ax\in Tx^+ $ and $ Ay\in Ty^- $. This completes the proof of the necessary part as $ M_T = \mathcal{M}_T $. 
\end{proof}

\begin{cor}\label{linearoperatorconnected}\cite[Theorem 2.1]{SP}
Let $ \mathbb{X} $ be a finite-dimensional real Banach space. Let $ T \in \mathbb{L}(\mathbb{X}) $ be such that $ T $ attains its norm at only $ \pm D $, where $ D $ is a connected subset of $ S_\mathbb{X} $. Then for $ A \in \mathbb{L}(\mathbb{X}) $ with $ T \perp_B A $ there exists $ x \in D $ such that $ Tx\perp_B Ax $.
\end{cor}

\begin{proof}
Considering the restriction on $ S_{\mathbb{X}}, $ it follows that  $ T ,A \in  \mathcal{C}(S_\mathbb{X}, \mathbb{X}) $ with $ T \perp_B A $. In view of Theorem \ref{realcontinuity}, there exists $ x, y \in \mathcal{M}_T $ such that $ Ax\in Tx^+ $ and $ Ay\in Ty^- $. Without loss of generality, let $ x\in D $. We claim that there exists $ y_0\in D $ such that $ Ay_0\in Ty_0^- $. Indeed, otherwise the linearity of $ T $ would imply $ Az\in Tz^+ $ for all $ z\in \mathcal{M}_T$, which contradicts Theorem \ref{realcontinuity}. Now, it follows from Theorem \ref{componentcontinuity} that there exists $ x\in \mathcal{M}_T$ such that $ Tx\perp_B Ax $. As before, note that $\mathcal{M}_T=M_T.$ This completes the proof of the corollary.
\end{proof}

We would like to characterize Birkhoff-James orthogonality of real bilinear forms on $ \mathbb{H}_1\times \mathbb{H}_2 $. To obtain the desired characterization, let us first present some useful facts about bilinear forms on $ \mathbb{H}_1\times \mathbb{H}_2 $. We omit the proofs as one can verify them directly.

\begin{prop}\label{preliminary}
Let $ \mathbb{H}_1,\mathbb{H}_2 $ be  finite-dimensional  Hilbert spaces and let $ \mathcal{T}_A, \mathcal{T}_B \in \mathcal{B}(\mathbb{H}_1\times \mathbb{H}_2, \mathbb{R}) $ be non-zero. Then the following hold true:\\
$ (i) $ $ \mathcal{T}_A + \lambda \mathcal{T}_B = \mathcal{T}_{A+\lambda B} $ for all $ \lambda \in \mathbb{R} $.\\
$ (ii) $ $ \|\mathcal{T}_A\| = \| A\| $.\\ 
$ (iii) $ $ M_{\mathcal{T}_A} = \{ (\pm \frac{Ay}{\|A\|}, y) : y \in M_A \}. $\\
\end{prop}

Now, the promised characterization:

\begin{theorem}\label{Bilinearform}
Let $ \mathbb{H}_1,\mathbb{H}_2 $ be finite-dimensional Hilbert spaces and let $ \mathcal{T}_A, \mathcal{T}_B \in \mathcal{B}(\mathbb{H}_1\times \mathbb{H}_2, \mathbb{R}) $ be non-zero. Then $ \mathcal{T}_A\perp_B \mathcal{T}_B $ if and only if there exists $ (x_0,y_0) \in M_{\mathcal{T}_A} $ such that $ \mathcal{T}_B(x_0,y_0) = 0 $.
\end{theorem}

\begin{proof}
We only prove the necessary part of the theorem, since the sufficient part  follows trivially. We observe that $ M_{\mathcal{T}_A} = W_1 \cup W_2 $, where
\begin{align*}
W_1 & = \left\{ \left(\frac{Ay}{\|A\|}, y\right) : y \in M_A \right \}~~and~~W_2  = \left\{ \left(-\frac{Ay}{\|A\|}, y\right) : y \in M_A \right\}.
\end{align*} 
If $ M_A = \{ \pm y \} $ for some $ y \in S_{\mathbb{H}_2} $, then the proof follows directly from Theorem \ref{componentcontinuity}, using the bilinearity of $ \mathcal{T}_A $. Otherwise, define $ f_1,f_2: \mathbb{H}_2 \rightarrow  \mathbb{H}_1\times \mathbb{H}_2 $ by $ f_1(y) = \left(\frac{Ay}{\|A\|}, y\right) $ and $ f_2(y) = \left(-\frac{Ay}{\|A\|}, y\right) $, respectively.  It follows from Theorem $2.2$ of \cite{SP} that $M_A$ is connected. Consequently, $ W_1 $ and $ W_2 $ are also connected, as they are continuous images of $ M_A $. Now, $ \mathcal{T}_A , \mathcal{T}_B \in \mathcal{C}(S_\mathbb{X}\times S_\mathbb{Y}, \mathbb{R}) $ with $ \mathcal{T}_A \perp_B \mathcal{T}_B $. Therefore in view of Theorem \ref{realcontinuity}, there exist $ (x_1,y_1),(x_2,y_2) \in M_{\mathcal{T}_A} $ such that $ \mathcal{T}_B(x_1,y_1)\in \mathcal{T}_A(x_1,y_1)^+ $ and $ \mathcal{T}_B(x_2,y_2)\in \mathcal{T}_A(x_2,y_2)^- $. Since $ \mathcal{T}_A, \mathcal{T}_B $ are real valued, we must have  $\mathcal{T}_A(x_1,y_1) \mathcal{T}_B(x_1,y_1) > 0 $ and $  \mathcal{T}_A(x_2,y_2)\mathcal{T}_B(x_2,y_2) < 0. $ Without loss of generality, suppose $ (x_1,y_1)\in W_1 $. Then there must be  $ (x_2,y_2)\in W_1 $ such that $ \mathcal{T}_A(x_2,y_2)\mathcal{T}_B(x_2,y_2) < 0 $. Indeed, otherwise the bilinearity of $ \mathcal{T}_A $ would imply $ \mathcal{T}_A(x,y)  \mathcal{T}_B(x,y) > 0 $ for all $ (x,y)\in M_{\mathcal{T}_A} $, which contradicts Theorem \ref{realcontinuity}. Now, the proof follows from Theorem \ref{componentcontinuity}.
\end{proof}

As a natural outcome of our present study, we may obtain a completely new and elementary proof of the classical Bhatia-\v{S}emrl Theorem in the real case, as an application of Theorem \ref{Bilinearform}. Note that in the following theorem, $A$ and $B$ are $n\times n$ matrices with real entries and $\mathbb{H}$ is an $n$-dimensional real Hilbert space for some $n>1$.

\begin{theorem}\label{BhatiaSemrl} A matrix $A$ is orthogonal to $B$ if and only if there exists a unit vector $ y_0 \in \mathbb{H} $ such that $||A||=||Ay_0||$ and $ \langle Ay_0, By_0 \rangle = 0 $.
\end{theorem}

\begin{proof}
Without loss of generality, let $ A $ be non-zero. We only prove the necessary part of the theorem as the proof of the sufficient part  follows trivially. Consider the bilinear forms $ \mathcal{T}_A, \mathcal{T}_B\in \mathcal{B}(\mathbb{H}\times \mathbb{H}, \mathbb{R}) $. Now, for any $ \lambda \in \mathbb{R} $, we have
\begin{align*}
\|\mathcal{T}_A + \lambda \mathcal{T}_B \| = \| A+ \lambda B \| \geq \| A\| = \| \mathcal{T}_A \|.
\end{align*}
Therefore, $ \mathcal{T}_A \perp_B \mathcal{T}_B $. In light of  Theorem \ref{Bilinearform},  there exists $ (x_0,y_0) \in M_{\mathcal{T}_A} $ such that $ \mathcal{T}_B(x_0,y_0) = 0 $. Also, since $ (x_0,y_0) \in M_{\mathcal{T}_A} $, we must  have $ x_0 = \pm\frac{Ay_0}{\|A\|} $ for some $ y_0\in M_A $. This implies that $ \mathcal{T}_B(x_0,y_0) = \pm \frac{1}{\|A\|} \langle Ay_0, By_0 \rangle = 0$. This establishes the theorem.
\end{proof}

Let us end this article with the following closing remark:

\begin{remark}
The study of Birkhoff-James orthogonality of linear operators between Hilbert spaces and Banach spaces have been conducted in detail in \cite{BG, BS, S, SP}. The main purpose of this article is to illustrate that it is possible to extend these results to the topological setting. Moreover, as a natural outcome of such a study, Birkhoff-James orthogonality of real bilinear forms can be characterized completely, which in turn allows us to obtain a new and elementary proof of the Bhatia-\v{S}emrl Theorem. Although a number of proofs of the said theorem are available in the literature \cite{ BG, BS, K, TA}, we believe that our proof is less complicated and differs from each of those in motivation.
\end{remark}

\end{document}